\documentclass[reqno, 12pt]{amsart}
\usepackage{amsmath}
\usepackage{amssymb}
\usepackage{amsthm}
\usepackage{mathtools}
\usepackage{tikz-cd}
\usepackage{enumerate}
\usepackage[mathscr]{eucal}
\usepackage{graphicx}
\usepackage{pgfplots}

\newtheorem{theorem}{Theorem}[section]
\newtheorem{lemma}[theorem]{Lemma}
\newtheorem{proposition}[theorem]{Proposition}
\newtheorem{corollary}[theorem]{Corollary}
\newtheorem{definition}[theorem]{Definition}
\newtheorem{note}[theorem]{Note}

\theoremstyle{definition}
\newtheorem{example}[theorem]{Example}

\newtheorem{remark}[theorem]{Remark}

\begin{document}
	\title[A study on the Dual spaces of $(C(X), \tau_{\mathcal{B}})$ and $(C(X), \tau_{\mathcal{B}}^s)$]{A study on the dual of $C(X)$ with the topology of (strong) uniform convergence on a bornology}	
	
	\author{Akshay Kumar}
	
	\newcommand{\acr}{\newline\indent}
	
	\address{Akshay Kumar: Department of Mathematics, Indian Institute of Technology Madras, Chennai-600036, India}
	\email{akshayjkm01@gmail.com}
	
	
	\thanks{This work is supported by IIT Madras (Project No. SB22231267MAETWO008573).}	
	
	
	\begin{abstract} This article begins by deriving a measure-theoretic decomposition of continuous linear functionals on $C(X)$, the space of all real-valued continuous functions on a metric space $(X, d)$, equipped with the topology $\tau_\mathcal{B}$ of uniform convergence on a bornology $\mathcal{B}$. We characterize the bornologies for which $(C(X), \tau_{\mathcal{B}})^*=(C(X), \tau_{\mathcal{B}}^s)^*$, where $\tau_{\mathcal{B}}^s$ represents the topology of strong uniform convergence on $\mathcal{B}$. Furthermore, we examine the normability of $\tau_{ucb}$, the topology of uniform convergence on bounded subsets, on $(C(X), \tau_{\mathcal{B}})^*$, and explore its relationship with the operator norm topology. Finally, we derive a topology on measures that shares a connection with $(C(X), \tau_{\mathcal{B}})^*$ when endowed with $\tau_{ucb}$. \end{abstract}

	\maketitle
	
	\section{Introduction and Preliminaries}
	\subsection{Introduction}
	Let $C(X)$ be the space of real-valued continuous functions on a metric space $(X,d)$. On $C(X)$, the classical topology of uniform convergence on a bornology $\mathcal{B}$, usually denoted by $\tau_{\mathcal{B}}$, has garnered significant attention from many mathematicians over the past several decades (see, \cite{Ucanbfms, Suc, SWasucob, MN}).

	It is known that if $\mathcal{B}=\mathcal{K}$, the collection of all non-empty relatively compact subsets of $X$, then $(C(X), \tau_k)$ forms a locally convex space. Moreover, the dual $(C(X), \tau_k)^*$ is isometrically isomorphic to the space of all closed regular finite Borel measures on $X$ with compact supports (see, Theorem 2.4.1, p. 88  in \cite{bns}). For a general bornology $\mathcal{B}$, $(C(X), \tau_{\mathcal{B}})$ may not necessarily form a topological vector space, as the scalar multiplication may fail to be jointly continuous. However, as shown in \cite{esaetvs}, $(C(X), \tau_{\mathcal{B}})$ is always an extended locally convex space (elcs), which allows for the study of its dual space.

	This paper is organized into four sections. Following a necessary preliminary details, the second section develops a decomposition of continuous linear functionals on $(C(X), \tau_{\mathcal{B}})$ in terms of closed, regular, finite Borel measures supported on $\mathcal{B}$.
	
	In the third section, we compare the dual spaces of $(C(X), \tau_{\mathcal{B}})$ and $(C(X), \tau_{\mathcal{B}}^s)$, where $\tau_{\mathcal{B}}^s$ denotes the topology of strong uniform convergence on the bornology $\mathcal{B}$. This topology was defined by Beer and Levi in \cite{Suc} and has been further explored by many people (see, \cite{ATasucob, Bspafs, SWasucob}). We provide an example demonstrating that the duals of these spaces may differ (see, Example \ref{dual of both spaces may not coincide}). Additionally, in Theorem \ref{condition under which both the spaces have same dual}, we find conditions under which both the spaces share the same dual.
	
	Since $(C(X), \tau_\mathcal{B})$ forms a locally convex space for $\mathcal{B}\subseteq\mathcal{K}$, the natural topology on $(C(X), \tau_{\mathcal{B}})^*$ is the topology $\tau_{ucb}$ of uniform convergence on bounded subsets of $(C(X), \tau_{\mathcal{B}})$ (see, Definition \ref{def of uniform convergence on bounded sets}). However, many people have considered the operator norm topology when studying $(C(X), \tau_{k})^*$ (see, \cite{bns, kundu1989spaces, wheeler1976mackey}). A natural question arises: do $\tau_{ucb}$ and $\|\cdot\|_{op}-$topology coincide? In the fourth section, we present Example \ref{example for operator and strong topology not coincide} to illustrate that this is not always the case. We further examine the normability of $\tau_{ucb}$ on $(C(X), \tau_{\mathcal{B}})^*$ and $(C(X), \tau_{\mathcal{B}}^s)^*$. We show that $\|\cdot\|_{op}$ may not define a norm on $(C(X), \tau_{\mathcal{B}})^*$. To address this, we introduce a topology $\sigma$ on measures, which shares a connection between $((C(X), \tau_{\mathcal{B}})^*, \tau_{ucb})$ and measures supported on $\mathcal{B}$.


	%
	%
	%
	%
	
	\subsection{Preliminaries}
	
	We denote a metric space by $(X, d)$ (or $X$). A collection $\mathcal{B}$ of non-empty subsets of $X$ is said to form a \textit{bornology} on $X$ if it covers $X$ and is closed under finite union and taking subsets of its members. A subfamily $\mathcal{B}_0$ of $\mathcal{B}$ is said to be a \textit{base} for $\mathcal{B}$ if it is cofinal in $\mathcal{B}$ under the set inclusion. If all the members of $\mathcal{B}_0$ are closed in $(X,d)$, then we say $\mathcal{B}$ has a \textit{closed base}. For more details on bornologies, we refer to \cite{bala, Ballf, hogbe,  Bc}.

	\begin{definition}\label{definition of tauB}{\normalfont (\cite{Suc})} \normalfont The topology $\tau_{\mathcal{B}}$ of \textit{uniform convergence on} $\mathcal{B}$ is determined by a uniformity on $C(X)$ having base consisting of sets of the form  $$[B, \epsilon]=\left\lbrace (f, g) : \forall x\in B,~ |f(x)-g(x)|<\epsilon \right\rbrace~ (B\in\mathcal{B},~ \epsilon>0).$$\end{definition}
	
	\noindent Note here that 
	\begin{enumerate}[(1)]
		\item if $\mathcal{B}=\mathcal{F}$, the bornology of all non-empty finite subsets of $X$, then $\tau_\mathcal{B}$ will give the \textit{topology of pointwise convergence}, which is usually denoted by $\tau_p$;
		
		\item if $\mathcal{B}=\mathcal{K}$, the boronology of all non-empty relatively compact subsets of $X$, then $\tau_\mathcal{B}$ will give the \textit{topology of uniform convergence on compact sets} or \textit{compact open topology}, which is usually denoted by $\tau_{k}$;
		
		\item if $\mathcal{B}=\mathcal{P}_0(X)$, the boronology of all non-empty subsets of $X$, then  $\tau_\mathcal{B}$ will give the \textit{topology of uniform convergence}.  We usually denote this topology by $\tau_u$.
	\end{enumerate}
	
	\noindent These particular cases have been widely explored in the literature (see, \cite{kelley, Fswufagt, willard}).
	
	For a bornology $\mathcal{B}$ on $X$, in \cite{Suc}, Beer and Levi present a variational form of $\tau_{\mathcal{B}}$ known as the topology of strong uniform convergence on $\mathcal{B}$. 
	
	\begin{definition}{\normalfont (\cite{Suc})} \normalfont The topology  $\tau^{s}_{\mathcal{B}}$ of \textit{strong uniform convergence on $\mathcal{B}$} is determined by a uniformity on $C(X)$ having base consisting of sets of the form  $$[B, \epsilon]^s=\left\lbrace (f, g) : \exists~ \delta>0 ~\forall x\in B^\delta,~ |f(x)-g(x)|<\epsilon \right\rbrace~ (B\in\mathcal{B},~ \epsilon>0), $$ where for $B\subseteq X$, $B^\delta=\displaystyle{\bigcup_{y\in B}}\{x\in X: d(x, y)<\delta\}$.  	 
	\end{definition}
	
	\begin{remark}\normalfont 
		\begin{enumerate}[(1)]
			\item For every bornology $\mathcal{B}$ on $X$, $\overline{\mathcal{B}}= \{A\subseteq \overline{B}: B\in \mathcal{B}\}$ forms a bornology on $X$. It is easy to prove that $\tau_{\mathcal{B}}=\tau_{\overline{\mathcal{B}}}$ and $\tau_{\mathcal{B}}^s=\tau_{\overline{\mathcal{B}}}^s$ on $C(X)$. Hence, from this point of view, we assume that every bornology $\mathcal{B}$ has a closed base $\mathcal{B}_0$.  	
			
			%
			\item  Note here that if $\mathcal{B}=\mathcal{F}$ and we extend our function space to $\mathbb{R}^X$, the set of all real-valued functions on $(X,d)$, then $\tau_{\mathcal{B}}^s$ will give the weakest topology on $\mathbb{R}^X$ for which $C(X)$ is closed (see, Corollary 6.8 in \cite{Suc}).
	\end{enumerate}	\end{remark}

	Recall from \cite{flctopology, esaetvs} that a vector space $Z$ equipped with a topology $\tau$ is said to be an \textit{extended locally convex space} (elcs) if $\tau$ is induced by a collection of extended seminorms (functions which satisfy all the properties of a seminorm and may attain the infinity value). Note that the topologies $\tau_{\mathcal{B}}$ and $\tau_{\mathcal{B}}^s$ on $C(X)$ are induced by the collections $\mathscr{P}=\{\rho_B: B\in \mathcal{B}_0\}$ and $\mathscr{P}^s=\{\rho_B^s: B\in \mathcal{B}_0\}$ of extended seminorms, respectively, where $\rho_B(f) =\sup_{x\in B}|f(x)|$ and $\displaystyle{\rho_B^s(f) =\inf_{\delta>0}\left\lbrace \sup_{x\in B^\delta}|f(x)|\right\rbrace}$ for all $f\in C(X)$ (see, \cite{flctopology}). Therefore, both the function spaces $(C(X), \tau_{\mathcal{B}})$ and $(C(X), \tau_{\mathcal{B}}^s)$ form extended locally convex spaces. One can note that both $\mathscr{P}$ and $\mathscr{P}^s$ are directed families as for every $B_1, B_2\in\mathcal{B}$, we have $B_1\cup B_2\in\mathcal{B}$. Consequently, $\max \{\rho_{B_1}, \rho_{B_2}\}\leq \rho_{B_1\cup B_2}$ and $\max \{\rho^s_{B_1}, \rho^s_{B_2}\}\leq \rho^s_{B_1\cup B_2}$. Therefore, $$\mathcal{N}=\{\rho_B^{-1}([0, r)): r>0 \text{ and } B\in\mathcal{B}\}\text{ and }$$ $$\mathcal{N}^s=\{(\rho_B^s)^{-1}([0, r)): r>0 \text{ and } B\in\mathcal{B}\}$$ give a neighborhood base at $0$ in $(C(X), \tau_{\mathcal{B}})$ and $(C(X), \tau_{\mathcal{B}}^s)$, respectively. For more details on extended locally convex spaces, we refer to \cite{doelcs, flctopology,  belcsaubp, Relcs, esaetvs}.
	
	\section{Measure theoretical Decomposition}
	
	Recall that $(C(X), \tau_{k})^*$ is isometrically isomorphic to the space of all closed regular Borel measures with compact support (see, Theorem 2.4.1, p. 88  in \cite{bns}). Therefore, a natural candidate to describe the dual of $(C(X), \tau_{\mathcal{B}})$ is the collection of all closed regular finite Borel measures with support in $\mathcal{B}$. Let
	\begin{enumerate}
		\item 	$\mathscr{B}(X)$ denotes the Borel sigma algebra induced by $X$;
		\item  $\mathscr{M}(X)$ denotes the collection of all finite Borel measures on $\mathscr{B}(X)$;
		\item $\mathscr{M}_{\mathcal{B}}(X)$ denotes the collection of all closed regular, finite, Borel measures on $\mathscr{B}(X)$ supported on $\mathcal{B}$, that is, for every $\mu\in\mathscr{M}_{\mathcal{B}}$, there exists a closed $B\in\mathcal{B}$ such that $|\mu|\left(X\setminus B\right)=0$.
	\end{enumerate}
	
	We use the next proposition to demonstrate that corresponding to every $\mu\in\mathscr{M}_{\mathcal{B}}(X)$, there exists a continuous linear functional on $(C(X), \tau_{\mathcal{B}})$. 
	
	For $B\in\mathcal{B}$, we define 
	\begin{eqnarray*}
		C(X)_{fin}^{B}&=& \{f\in C(X): \rho_B(f)<\infty\}\\
		&=&\{f\in C(X): \sup_{x\in B}|f(x)|<\infty\}, \text{ and }\\
		C(X)_{fin}^{B_s}&=& \{f\in C(X): \rho_B^s(f)<\infty\}\\
		&=&\left\lbrace f\in C(X): \inf_{\delta>0}\left(\sup_{x\in B^\delta}|f(x)|\right)<\infty\right\rbrace.  
	\end{eqnarray*}
	
	\begin{proposition}\label{every measure gives a continuous map} Suppose $\mu\in \mathscr{M}_{\mathcal{B}}(X)$ is supported on closed $B\in\mathcal{B}$. Then the linear functional $H_\mu(f)=\int_X fd\mu$ for all $f\in C(X)_{fin}^B$ is continuous on $\left(C(X)_{fin}^B, \tau_\mathcal{B}|_{C(X)_{fin}^B}\right)$.  
	\end{proposition}
	
	\begin{proof} For every $f\in C(X)_{fin}^B$, we have
		\begin{eqnarray*}
			|H_\mu(f)|&=& \left|\int_X fd\mu\right|\leq \int_X|f|d\mu\\
			& =&\int_B|f|d\mu\leq \sup_{x\in B}|f(x)| \int_{B}1 d\mu\\
			& \leq& \rho_B(f)|\mu|(B).
		\end{eqnarray*}  Hence, $H_\mu$ is continuous on $C(X)_{fin}^B.$
	\end{proof}
	
	\begin{lemma}{\normalfont(\cite{flctopology, esaetvs})} For every continuous extended seminorm $\rho$ on an elcs $(Z, \tau)$, the following statements hold.
		\begin{enumerate}[(1)]
			\item $Z_{fin}^\rho:=\{x\in Z: \rho(x)<\infty\}$ is a clopen subspace of $(Z, \tau)$.
			\item If $Z=Z_{fin}^\rho\oplus M$ (algebraic direct sum) for some subspace $M$ of $Z$, then $M$ is a discrete subspace of $(Z, \tau)$.  
		\end{enumerate}
	\end{lemma}	
	
	Observe here that for every $B\in \mathcal{B}$, $\rho_B$ gives a continuous extended seminorm on $(C(X), \tau_{\mathcal{B}})$. Therefore, $C(X)_{fin}^{B}$ is a clopen subspace of $(C(X), \tau_{\mathcal{B}})$ and $M_B$ is a discrete subspace of $(C(X), \tau_{\mathcal{B}})$, where $C(X)=C(X)_{fin}^B\oplus M_B$ (it is an algebraic direct sum. However, it will turn out to be a topological direct sum). Consequently, the linear map $H_\mu\circ p_B$ is continuous on $(C(X), \tau_{\mathcal{B}})$, where $p_B$ is the projection of $C(X)$ onto $C(X)_{fin}^B$ and $H_\mu$ is the linear map as defined in Proposition \ref{every measure gives a continuous map}. Hence, for every closed regular, finite Borel measure $\mu$ with a support in $\mathcal{B}$, we have a continuous linear functional $H_\mu\circ p_B$ on $(C(X), \tau_{\mathcal{B}})$. Since $M_B$ is  discrete subspace of $(C(X), \tau_{\mathcal{B}})$, every linear functional on $M_B$ is continuous. Thus, for every $\mu\in\mathscr{M}_{\mathcal{B}}(X)$ supported on $B$ and linear functional $H_2$ on $M_B$, the linear functional $F:= H_\mu\circ p_B+H_2\circ q_B$ is continuous on $(C(X), \tau_{\mathcal{B}})$, where $q_B$ is the projection of $C(X)$ onto $M_B$. We next show that every continuous linear functional on $(C(X), \tau_{\mathcal{B}})$ has this form. \\
	
	\noindent We now define support of a continuous linear functional on $(C(X), \tau_{\mathcal{B}})$.
	\begin{definition}\normalfont  We say a linear functional $F$ on $C(X)$ is \textit{supported  on} $B\in\mathcal{B}$ if $F(f)=0$ whenever $f|_B=0$.\end{definition}   
	
	\begin{lemma}\label{support of a continuous linear functional} Suppose $F$ is a continuous linear functional on $(C(X), \tau_{\mathcal{B}})$. Then there exists a closed $B\in \mathcal{B}$ such that $F$ is supported on $B$.\end{lemma} 
	
	\begin{proof} Since $F$ is continuous on $(C(X), \tau_{\mathcal{B}})$, there exists a $B\in \mathcal{B}_0$ such that $|F(f)|\leq \rho_B(f)$ for all $f\in C(X)$. It is easy to see that $F$ is supported on $B$.
	\end{proof}
	
	The next example shows that a continuous linear functional on $(C(X), \tau_{\mathcal{B}})$ can be supported on two distinct sets.
	\begin{example} Let $\mathcal{B}=\{A\subseteq [a, \infty): a\in \mathbb{R}\}$. Consider the linear functional $F$ on $(C(X), \tau_{\mathcal{B}})$ by $$F(f)=\int_{0}^{1}f(x)dx.$$ Then $F\in (C(X), \tau_{\mathcal{B}})^*$. It is easy to see that $F$ is supported on both $[0, 1]$ and $[0, 4]$. 
	\end{example}
	
	\begin{note}\normalfont One can note here that if $F$ is supported on $A$, then $F$ is supported on $D$ for every $A\subseteq D.$\end{note}  
	
	Recall that one of the important steps in examining barreledness and many other functional properties of $(C(X), \tau_{k})$ is to show that every continuous linear functional on $(C(X), \tau_{k})$ has minimal support in $\mathcal{K}$ (see, \cite{bns}). The next example shows that this may not be true in the case of an arbitrary bornology. However, for a non-zero $F\in (C(X), \tau_{\mathcal{B}})$, the collection $\mathscr{S}_F:=\{B\in\mathcal{B}: F \text{ is supported on } B \}$ always satisfies the finite intersection property.
	
	\begin{proposition}\label{F is supported on A intersection B} Let $0\neq F\in \left(C(X), \tau_{\mathcal{B}}\right)^*$ supported on $A, B\in\mathcal{B}$. Then $F$ is supported on the non-empty set $A\cap B$. \end{proposition}
	
	\begin{proof} We first show that $A\cap B\neq \emptyset$. Let $A\cap B=\emptyset$. By Pasting lemma and Tietze extension theorem, for each $f\in C(X)$, we have an $h\in C(X)$ such that $h=f$ on $A$ and $h=0$ on $B$. Since $F$ is supported on both $A$ and $B$, $F(f)=F(h)=0$. Thus, $F(f)=0$ for all $f\in C(X)$. Which is not possible. Now, suppose $f|_{A\cap B}=0$. Define $g_0:A\cup B\to \mathbb{R}$ by $g_0=f$ on $A$ and $g_0=0$ on $B$. Let $g$ be a continuous extension of $g_0$ on $X$. Since $g|_B=0$ and $f-g|_A=0$ and $F$ is supported on both $A$ and $B$, we have $F(f)=F(g)=0$. Hence, $F$ is supported on $A\cap B$. 
	\end{proof}
	
	\begin{example} Let $\mathcal{B}=\{A\subseteq [a, \infty): a\in \mathbb{R}\}$. Define a continuous linear functional $F$ on $(C(X), \tau_{\mathcal{B}})$ by $F(f)=\int_{-1}^{0}q_B(f)dx$, where $B=[0, \infty)$. Then $F$ is supported on $[a, \infty)$ for all $a\in\mathbb{R}$ (as $[0, a]$ is compact for every $a>0$). Hence, $\bigcap \{B\in \mathcal{B}_0: F \text{ is supported on } B\}=\emptyset$.   
	\end{example} 
	
	We need the following lemma in order to derive the main theorem of this section.   
	\allowdisplaybreaks
	\begin{lemma}\label{2} Suppose $F\in (C(X), \tau_{\mathcal{B}})^*$. Then there exists a closed $B\in\mathcal{B}$ and a unique $F_B\in (C_b(B), \parallel\cdot\parallel_\infty)^*$ such that $F_B(f|_B)=F(f)$ for all $f\in C(X)_{fin}^B$, where $C_b(B)$ is the collection of all bounded, continuous functions on $B$ and $\parallel h\parallel_\infty=\sup_{x\in B}|h(x)|$ for all $h\in C_b(B)$. \end{lemma}
	
	\begin{proof} Let $B\in \mathcal{B}$ such that $|F|\leq \rho_B$. Define $F_B(h)=F(\hat{h})$ for all $h\in C_b(B)$, where $\hat{h}$ is any continuous extension of $h$ on $X$. Since $F$ is supported on $B$, $F_B$ is well-defined and unique. Note that $|F_B(f)|\leq \rho_B(f)=\sup_{x\in B}|f(x)|=\|f\|_\infty$ for all $f\in C_b(B)$. Thus, $F_B\in \left(C_b(B), \parallel\cdot\parallel_\infty\right)^*$. \end{proof}

	If $B\in\mathcal{B}$ and $C(X)=C(X)_{fin}^B\oplus M_B$, then we denote $p_B$ and $q_B$ by the projections of $C(X)$ onto $C(X)_{fin}^B$ and $M_B$, respectively.
	
	\begin{theorem}\label{decomposition of continuous linear functionals} Suppose $F\in (C(X), \tau_{\mathcal{B}})^*$. Then there exists a closed set $B\in \mathcal{B}$ such that $F=H_\mu\circ p_B+ F\circ q_B$ for some $\mu\in\mathscr{M}_{\mathcal{B}}$ supported on $B$, that is, $F(f)=\int_X fd\mu$ for all $f\in C(X)_{fin}^B$. Moreover, $\|\mu\|=\sup\{|F(f)|: \|f\|_\infty\leq 1\}$, where $\|\mu\|$ represents the total variation of $\mu$ which is defined by $\|\mu\|:=|\mu|(X)$.    
	\end{theorem}
	
	\begin{proof}By Lemma \ref{support of a continuous linear functional}, $F$ is supported on some closed $B\in\mathcal{B}$. By Lemma \ref{2}, there exists a unique $F_B\in (C_b(B), \parallel\cdot\parallel_\infty)^*$ such that $F_{B}(f|_B)=F(f)$ for all $f\in C(X)_{fin}^B$. There exists a closed regular, finite sigma measure $\nu$ on $\mathscr{B}(B)$ such that $F_B(f|_B)=\int_{B}f|_Bd\nu$ for all $f\in C(X)_{fin}^B$. Since $\mathscr{B}(B)= \{A\cap B: A\in \mathscr{B}(X)\}$, $\mu(A):=\nu(A\cap B)$ defines a finite measure on $\mathscr{B}(X)$. It is easy to see that $\mu$ is supported on $B$. We now show that $\mu$ is closed regular. Let $A\in \mathscr{B}(X)$. Then $A\cap B\in \mathscr{B}(B)$. Therefore, for $\epsilon>0$, there exist a closed set $C$ and an open set in $U$ in $B$ such that $C\subseteq A\cap B\subseteq U$ and $|\nu|(U\setminus C)<\epsilon$. Since $B$ is closed in $X$, $C$ is closed in $X$. Take $V=U\cup (X\setminus B)$. Then $V$ is open in $X$. It is easy to see that  $C\subseteq A\subseteq V$, and $|\mu|(V\setminus C)=|\mu|((U\setminus C)\cap B)=|\nu|(U\setminus C)<\epsilon.$ Thus, $\mu$ is closed regular. Note that $F(f)=F_B(f|_B)=\int_Bf|_Bd\nu=\int_Xfd\mu$ for all $f\in C(X)_{fin}^B$. Hence, $F=H_\mu\circ p_B+H_2\circ q_B$ for $H_2=F|_{M_B}$.
		
		It is easy to see that $\sup\{|F(f)|: \|f\|_\infty\leq 1\}\leq |\mu|(X)= |\mu|(B)$. Since $\mu$ is closed regular, for $\epsilon>0$, we have a closed set $C$ and an open set $U$ in $X$ such that $C\subseteq B\subseteq U$ and $|\mu|\left(U\setminus C\right) <\epsilon$. Since $C\cap \left(X\setminus U\right)=\emptyset$, we have a continuous function $h$ on $X$ such that $h|_{C}= \frac{\epsilon}{|\mu(B)|+1}$, $h|_{X\setminus U}=0$, and $\|h\|_\infty \leq 1$.  Note that 
		\begin{eqnarray*}
			\left|\int_X hd\mu-|\mu|(B)\right|&\leq& \left|\int_C hd\mu+\int_{U\setminus C} hd\mu-|\mu|(B)\right|\\
			&\leq& \frac{\epsilon|\mu|(C)}{|\mu|(B)+1}+|\mu|\left(U\setminus C\right)+|\mu|\left(U\setminus C\right)\\
			&\leq& \frac{\epsilon|\mu|(B)}{|\mu|(B)+1}+2|\mu|\left(U\setminus C\right) \leq  3\epsilon.
		\end{eqnarray*}
		Thus, $\left|\int_Xhd\mu-|\mu|(B)\right|\leq 3\epsilon$. Consequently, 	
		\begin{eqnarray*}
			|\mu|(B)+3\epsilon &\leq& \left|\int_Xhd\mu\right|\\
			&\leq& \sup\{|F(f)|: \|f\|_\infty\leq 1\}.  
		\end{eqnarray*}
		Hence, $\|\mu\|=\sup\{|F(f)|: \|f\|_\infty\leq 1\}$.  \end{proof}

	\begin{remark}Suppose $\mathcal{B}=\mathcal{K}$, the collection of all non-empty compact subsets of $X$. Then for each $B\in\mathcal{B}$, $\rho_B$ is a seminorm on $C(X)$. Which implies that $M_B=\{0\}$. Consequently, $q_B=0$. Hence, every continuous linear functional $F$ on $(C(X), \tau_k)$ has the form $F=H_\mu$ for some finite Borel measure $\mu$ supported on some compact subset of $X$. Since $C_b(X)$ is dense in $(C(X), \tau_{\mathcal{B}})$, the function $$\|F\|_{op}=\sup\{|F(f)|: \|f\|_\infty\leq 1\}$$ defines a norm on $(C(X), \tau_{\mathcal{B}})^*$. As $\|F\|_{op}=\|\mu\|$, $((C(X), \tau_{\mathcal{B}})^*, \|\cdot\|_{op})$ is isometrically isomorphic to $(\mathscr{M}_{\mathcal{K}}, \|\cdot\|)$.   
	\end{remark}
	
	\begin{corollary} Suppose $\mathcal{B}\subseteq \mathcal{K}$. Then $((C(X), \tau_{\mathcal{B}})^*, \|\cdot\|_{op})$ is isometrically isomorphic to $(\mathscr{M}_{\mathcal{B}}, \|\cdot\|)$.
	\end{corollary}	
	
	\section{Comparison between the duals of $(C(X), \tau_{\mathcal{B}})$ and $(C(X), \tau_{\mathcal{B}}^s)$}
	
	As mentioned earlier an extended locally convex space (elcs) may not form a topological vector space. Therefore, we may not directly apply the classical techniques of locally convex spaces to these new spaces. To address this problem, the authors in \cite{flctopology} constructed the finest locally convex topology for an elcs which is still coarser than the given extended locally convex space topology (they called this topology by flc topology ($\tau_F$)). It is also shown in the same paper that if $\tau_F$ and $\tau_F^s$ are the flc topologies for $(C(X), \tau_{\mathcal{B}})$ and $(C(X), \tau_{\mathcal{B}}^s)$, respectively, then $\tau_F$ coincides with $\tau_F^s$ if and only if $\mathcal{B}$ is shielded from closed sets. A natural question one can ask here is when both the spaces have the same weak topologies? Which is equivalent to asking when both the spaces have the same dual? We show that the inclusion $(C(X), \tau_{\mathcal{B}}) \subseteq (C(X), \tau_{\mathcal{B}}^s)$ may be strict. We also find the condition under which both the spaces have the same dual.   
	
	\begin{lemma}\label{support of a continuous linear functional on tauBs} Suppose $0\neq F\in (C(X), \tau_{\mathcal{B}}^s)^*$. Then there exists a closed $B\in \mathcal{B}$ such that $F$ is supported on $B^\delta$ for every $\delta>0$. 
	\end{lemma} 
	
	\begin{proof} It is similar to the proof of Lemma \ref{support of a continuous linear functional}.
	\end{proof}

	Recall from \cite{Suc} that a superset $A_1$ of $A$ is said to be a \textit{shield} for $A$ if for every non-empty closed set $C\subseteq X$ with $C\cap A_1=\emptyset$, $C$ is not near to $A$, that is, $C\cup A^\delta=\emptyset$ for some $\delta>0$. A bornology $\mathcal{B}$ is said to be \textit{shielded from closed sets} if every member of $\mathcal{B}$ has a closed shield in $\mathcal{B}$.  
	
	The next example shows that, in general, $(C(X), \tau_{\mathcal{B}})^*$ may not coincide with $(C(X), \tau_{\mathcal{B}}^s)^*$. 
	
	\begin{example}\label{dual of both spaces may not coincide}\normalfont Suppose $X=\mathbb{R}$ and $\mathcal{B}$ is the bornology induced by the base $$\mathcal{B}_0=\{\mathbb{N}\cup A: A \text{ is any non-empty finite subset of } \mathbb{R}\}.$$ Note that if $A$ is any finite subset of $\mathbb{R}$. Then there exists a $k>3$ such that $A\cap [k, \infty)=\emptyset$. Since an arbitrary union of a locally finite family of closed sets is closed (see, Corollary 1.1.12, p. 17 in \cite{GTE}), $C:=\bigcup_{n\geq k}\left[n+\frac{1}{n}, n+1-\frac{1}{n}\right]$ is closed in $\mathbb{R}$. It is easy to prove that $C\cap (\mathbb{N}\cup A)=\emptyset$ and $\mathbb{N}^\delta \cap C\neq \emptyset$ for all $\delta>0$. So, $\mathcal{B}$ is not shielded from closed sets. Thus, $\tau_{\mathcal{B}}\subsetneq \tau_{\mathcal{B}}^s$ on $C(X)$.
		
		We now construct an $F\in (C(X), \tau_{\mathcal{B}}^s)^*\setminus(C(X), \tau_{\mathcal{B}})^*$.  Let $B=\mathbb{N}$. Then $\rho_B^s$ is continuous on $(C(X), \tau_{\mathcal{B}}^s)$. Consequently, $C(X)_{fin}^{B_s}=\{f\in C(X): (\rho_B^s)(f)<\infty\}$ is clopen in $(C(X), \tau_{\mathcal{B}}^s)$. Define an $h\in C(X)$ by 
		\[ h(x)=
		\begin{cases}
		\text{$0,$} &\quad\text{if $x\leq 1$}\\
		\text{$n(n+1)(x-n)$,} &\quad\text{if $n\leq x\leq n+\frac{1}{n+1}$ for $n\in \mathbb{N}$}\\
		\text{$(n+1)(n+1-x),$} &\quad\text{if $n+\frac{1}{n+1}\leq x\leq n+1$ for $n\in\mathbb{N}$}.\\
		\end{cases}\]	
		
		Clearly, $h(n)=0$ for all $n\in \mathbb{N}$. Suppose $\delta>0$. Then for every positive integer $m$ with $\frac{1}{m}<\delta$, we have $m+\frac{1}{m+1}\in B^\delta$ and $h(m+\frac{1}{m+1})=m$. Consequently, $\sup_{x\in B^\delta}|h(x)|=\infty$. Since this holds for all $\delta>0$, $\rho_B^s(h)=\infty$. Therefore, $h\notin C(X)_{fin}^{B_s}$. Suppose $M_{B_s}$ is any algebraic complement of $C(X)_{fin}^{Bs}$ containing $h$. Then $C(X)=C(X)_{fin}^{B_s}\oplus M_{B_s}$ becomes a topological direct sum. Consider the linear function $F$ defined by $$F(f)=\int_{1}^{\frac{3}{2}}q_{B_s}(f)(x)dx \text{ for all } f\in C(X),$$ where $q_{B_s}$ is the projection of $C(X)$ on $M_{B_s}$. Note that $|F(f)|\leq \rho_B^s(f)$ for all $f\in C(X)$. Therefore, $F\in (C(X), \tau_{\mathcal{B}}^s)^*$. It should be noted that $F$ is supported on $\mathbb{N}^\delta$ for every $\delta>0$. Let $F\in (C(X), \tau_{\mathcal{B}})^*$. By Lemma \ref{support of a continuous linear functional}, there exists a finite set $A\subseteq \mathbb{R}$ such that $F$ is supported on $\mathbb{N}\cup A$. Consider a bounded function $g\in C(X)$ such that $g|_{A\cup \mathbb{N}}=h|_{A\cup\mathbb{N}}$. It is easy to see that $g\in C(X)_{fin}^{B_s}$ and $g-h|_{A\cup\mathbb{N}}=0$. Since $q_{B_s}(g-h)=-h$, $F(g-h)=-\int_{1}^{\frac{3}{2}}h(x)dx\neq 0$. We arrive at a contradiction. 
	\end{example} 
	
	\begin{theorem}\label{condition under which both the spaces have same dual} Suppose $\mathcal{B}$ is a bornology on $X$ with a closed base $\mathcal{B}_0$. Then the following statements are equivalent:
		\begin{enumerate}[(1)]
			\item $\tau_{\mathcal{B}}=\tau_{\mathcal{B}}^s$ on $C(X)$;
			\item $\mathcal{B}$ is shielded from closed sets;
			\item $(C(X), \tau_{\mathcal{B}})^*=(C(X), \tau_{\mathcal{B}}^s)^*$.
	\end{enumerate}	\end{theorem}
	
	\begin{proof} The implications (1)$\Leftrightarrow$(2)$\Rightarrow$(3) follow from Theorem 4.1 in \cite{Suc}.
		
		(3)$\Rightarrow$(2). Suppose $B_0\in\mathcal{B}$ does not have any closed shield in $\mathcal{B}$, that is, for all closed set $B\in\mathcal{B}$ with $B_0\subseteq B$, there exists a closed set $C_B$ such that $C_B\cap B=\emptyset$ and $C_B\cap B_0^\delta \neq \emptyset$ for all $\delta >0$. Let $C(X)=C(X)_{fin}^{B_0s}\oplus M_{B_0s}$. Consider the linear functional $F$ on $C(X)$ defined by $F(f):=G(q_{B_0s}(f)),$ where $G$ is any linear functional on $M_{B_0s}$ such that $G(h)=0$ only if $h=0$. Note that $|F(f)|\leq \rho_{B_0}^s(f)$ for all $f\in C(X)$. Thus, $F\in (C(X), \tau_{\mathcal{B}}^s)^*$. By our hypothesis, $F\in (C(X), \tau_{\mathcal{B}})^*$. By Lemma \ref{support of a continuous linear functional}, $F$ is supported on some closed set $A\in\mathcal{B}$. Take $D=A\cup B_0$. Clearly, $D\in\mathcal{B}$ and $B_0\subseteq D$. Since $B_0$ does not have any closed shield in $\mathcal{B}$, there exists a closed set $C_D$ such that $C_D\cap D=\emptyset$ and $C_D\cap B_0^{\delta}\neq \emptyset$ for all $\delta>0$. Let $x_n\in C_D\cap B_0^{\frac{1}{n}}$. Note that if $x_{n_k}$ is any subsequence of $(x_n)$ which converges to $x$, then there exists a sequence $(b_{n_k})$ in $B_0$ such that $d(x_{n_k}, b_{n_k})<\frac{1}{n_k}\to 0$. Thus, $x\in B_0\cap D$. Which is not possible. Therefore, no subsequence of $(x_n)$ is convergent in $X$. Thus,   $P=\{x_n:n\in\mathbb{N}\}$ is a closed, discrete subset of $X$. We can always find an $h\in C(X)$ such that $h|_D=0$ and $h(x_n)=n$ for all $n\in\mathbb{N}$. Consequently, $h|_A=0$ and $q_{B_0s}(h)\neq 0$. Thus, $h|_A=0$ and $F(h)\neq 0$. We arrive at a contradiction as $F$ is supported on $A$. \end{proof}

	\section{Topology of uniform convergence on bounded sets}
	
	In section 2, we have noted that $((C(X), \tau_k)^*, \|\cdot\|_{op})$ is isometrically isomorphic to $(\mathscr{M}_{\mathcal{K}}, \|\cdot\|)$, where $\tau_k$ is the compact open topology on $C(X)$.  However, the $\|\cdot\|_{op}$ shows no relation with $\mathcal{B}$. In the case of a Hausdorff locally convex space $(Z, \tau)$,  the generalization of the operator norm topology is the strong topology, which is the topology of uniform convergence on bounded subsets of $Z$. The suitable adaptation of this topology for an elcs has been studied in \cite{doelcs, belcsaubp, Relcs} (they denoted this topology by $\tau_{ucb}$).

	In this section, we first compare $\tau_{ucb}$ with $\|\cdot\|_{op}$. We then identify the conditions under which $\tau_{ucb}$ is normable, and derive a topology on $\mathscr{M}_\mathcal{B}$ that shares a connection with $(C(X), \tau_{\mathcal{B}})^*$ when endowed with $\tau_{ucb}$. We begin with the following definition.        
	
	\begin{definition}\normalfont(\cite{doelcs})	Suppose $(Z, \tau)$ is an elcs. Then $A\subseteq Z$ is said to be \textit{bounded} in $(Z, \tau)$ if for every neighborhood $U$ of $0$, there exist an $r>0$ and a finite set $F\subseteq Z$ such that $A\subseteq F+rU$. \end{definition}
	
	\noindent The following points about bounded subsets of an elcs $(Z, \tau)$ are either easy to verify or given in \cite{doelcs}.
	\begin{enumerate}[(1)]
		\item Every finite subset of $Z$ is bounded.
		\item If $(Z, \tau)$ is a locally convex space, then $A\subseteq Z$ is bounded if and only if for every neighborhood $U$ of $0$, there exists an $r>0$ such that $A\subseteq rU$.
		\item The collection of all bounded sets forms a bornology on $Z$.
		\item If $A$ is bounded, then for all $f\in Z^*$, $f(A)$ is bounded. The converse of this holds if $(Z, \tau)$ forms a locally convex space. 
	\end{enumerate}

	\begin{definition}\label{def of uniform convergence on bounded sets}\normalfont(\cite{doelcs}) Let $(Z, \tau)$ be an elcs. Then the topology $\tau_{ucb}$, on $Z^*$, of \textit{uniform convergence on bounded subsets of} $(Z, \tau)$ is induced by the collection $\mathcal{P}=\{\rho_A: A ~\text{is a bounded subset of}~ (Z, \tau)\}$ of seminorms on $Z^*$, where $\rho_A(F)=\sup_{x\in A}|F(x)| ~\text{for all}~ x\in Z^*$.\end{definition}
	
	\noindent The following points about $\tau_{ucb}$ are given in \cite{doelcs}. 
	\begin{enumerate}[(1)]
		\item The space $(Z^*, \tau_{ucb})$ forms a locally convex space.
		\item  The collection $\mathcal{N}=\{A^\circ: A \text{ is bounded in } (Z, \tau)\}$ gives a neighborhood base at $0$ in $(Z^*, \tau_{ucb})$, where $A^\circ:=\{F\in Z^*: |F(x)|\leq 1 \text{ for all } x\in A\}$.
		\item For an lcs $(Z, \tau)$, $\tau_{ucb}$ coincides with its strong topology. 
	\end{enumerate}
	
	\begin{proposition}\label{op forms a norm}\normalfont The map $\|\cdot\|_{op}$ defines a norm on $(C(X), \tau_{\mathcal{B}})^*$ if and only if $\mathcal{B}\subseteq \mathcal{K}$. 
	\end{proposition}
	\begin{proof} Suppose $B\in\mathcal{B}$ is not compact. Then $C(X)_{fin}^B$ is a proper open subspace of $(C(X), \tau_{\mathcal{B}})$. Let $C(X)=C(X)_{fin}^B\oplus M_B$. Consider a linear functional $F$ on $C(X)$ such that $F(f)= 0$ for all $f\in C(X)_{fin}^B$ and $F(f)\neq 0$ for some non-zero $f\in M_B$. Clearly, $F$ is a non zero continuous functional on $(C(X)_{fin}^B, \tau_{\mathcal{B}})$ as $|F(f)|\leq \rho_B(f)$ for all $f\in C(X)$. Since $C_b(X) \subseteq C(X)_{fin}^B$, $\|F\|_{op}=0$. Conversely, if $\mathcal{B}\subseteq\mathcal{K}$, then $C_b(X)$ is dense in $(C(X), \tau_{\mathcal{B}})$. Thus, $\|\cdot\|_{op}$ defines a norm on $(C(X), \tau_{\mathcal{B}})^*$.    	   
	\end{proof}
	
	The natural question one can ask here are both the topologies $\tau_{ucb}$ and $\|\cdot\|_{op}$ coincide for $\mathcal{B}=\mathcal{K}$? The next example shows that, in general, it may not always be the case.

	\begin{example}\label{example for operator and strong topology not coincide} Suppose $X=\mathbb{R}$ and $\mathcal{B}=\mathcal{K}$. For each $m\in\mathbb{N}$, consider the functions $f_m$ defined by 
		\[ f_m(x)=
		\begin{cases}
		\text{$0,$} &\quad\text{if $x<0$}\\
		\text{$x$,} &\quad\text{if $0\leq x\leq m$ }\\
		\text{$m,$} &\quad\text{if $x>m$}\\
		\end{cases}\]
		%
		%
		%
		%
		%
		Note that if $K=[a, b]$, then $\sup_{x\in K}|f_m(x)|\leq \sup_{x\in K}|f_{m_0}(x)|\leq m_0$ for $m_0>b$. Therefore, $\mathbb{A}=\{f_m:m\in\mathbb{N}\}$ is bounded in $(C(X), \tau_{\mathcal{B}})$. Thus, $\mathbb{A}^\circ$ gives a neighborhood at $0$ in $\left((C(X), \tau_{\mathcal{B}})^*, \tau_{ucb}\right)$. Let $\tau_{ucb}=\tau_{\|\cdot\|_{op}}$. Then there exists an $r>0$ such that $B_{op}[0, r]\subseteq \mathbb{A}^\circ$, where $B_{op}[0, r]=\{F\in (C(X), \tau_{\mathcal{B}})^*: \|F\|_{op}\leq r\}$. Consequently, $\mathbb{A}\subseteq (\mathbb{A}^\circ)_\circ\subseteq (B_{op}[0, r])_\circ$, where  $(\mathbb{A}^\circ)_\circ= \{f\in C(X): |F(f)|\leq 1 \text{ for all } F\in \mathbb{A}^\circ\}$. Which is not true as for $k\in\mathbb{N}$ with $kr>1$, the continuous linear functional $F(f):=\int_{k}^{k+r}f(x)dx$ belongs to $B_{op}[0, r]$, but $|F(f_k)|=\int_{k}^{k+r}f_k(x)dx=\int_{k}^{k+r}kdx=kr>1$. 	\end{example}
	
	We next find conditions under which the topology $\tau_{ucb}$ is normable on $(C(X), \tau_{\mathcal{B}}^s)^*$ (or $(C(X), \tau_{\mathcal{B}})^*$). We use the following well known results in the sequel. 
	\begin{enumerate}[(1)]
		\item A locally convex space topology is normable if and only if it has a bounded neighborhood at $0$ (see, Theorem 6.2.1, p. 160 in \cite{tvsnarici}).
		\item If $(Z, \tau)$ is an elcs and $A^\circ$ is bounded in $(Z^*, \tau_{ucb})$, then $A$ is absorbing in $Z$ (it follows from Theorem 3.2 in \cite{flctopology} and Theorem 8.3.5, p. 234 in \cite{tvsnarici}).  
	\end{enumerate} 
	
	\begin{proposition}\label{normability of strong topology} Suppose $\tau_{ucb}$ on $(C(X), \tau_{\mathcal{B}}^s)^*$ $(\text{or } (C(X), \tau_{\mathcal{B}})^*)$ is normable. Then $\mathcal{B}\subseteq \mathcal{K}$. \end{proposition}
	
	\begin{proof} Let $\mathbb{B}$ be a closed, convex, and bounded subset of $(C(X), \tau_{\mathcal{B}}^s)$ such that $\mathbb{B}^\circ$ gives a bounded neighborhood at $0$ in $((C(X), \tau_{\mathcal{B}}^s)^*, \tau_{ucb})$. Consequently, $\mathbb{B}$ is an absorbing subset of $C(X)$. Let $A\in \mathcal{B}$ be closed. Then $\mathbb{B}\subseteq \rho_A^{-1}([0, r))+\mathbb{F}$ for some $r>0$ and finite subset $\mathbb{F}$ of $C(X)$. Therefore, $\mathbb{B}\subseteq C(X)_{fin}^B\oplus M$ for some finite dimensional subspace $M$ of $C(X)$. If $A$ is not compact, then there exists a closed, discrete subset $D=\{x_k:k\in\mathbb{N}\}$ of $A$. For $n\in\mathbb{N}$, consider the continuous function $h_n:X\to \mathbb{R}$ such that 	\[ h_n(x_k)=
		\begin{cases}
		\text{$0,$} &\quad\text{when $k<n$}\\
		\text{$k^{2n},$} &\quad\text{when $k\geq n$}\\
		\end{cases}\] Note that if $\sum_{j=1}^{m}\alpha_j h_{n_j}=0$, then for all $1\leq j\leq m$, $\sum_{j=1}^{m}\alpha_j h_{n_j}(x_{n_j})=\sum_{i=j}^{m}\alpha_i n_j^{2n_i}=0$. Consequently, $\alpha_j=0$ for all $j$. Thus, the set $\mathbb{K}=\{h_n: n\in\mathbb{N}\}$ is linearly independent. Since $\mathbb{B}$ is absorbing, $\mathbb{K}\subseteq C(X)_{fin}^B\oplus M$. We next show that $\text{span}(\mathbb{K})\cap C(X)_{fin}^B=\{0\}$. Let $h:=\sum_{i=1}^{m}\alpha_i h_{i}$ for $\alpha_m>0$ and $m>1$. Then for $k>\max\{m^2, \frac{|\alpha_i|}{\alpha_m}: 1\leq i\leq m \}$, we have 
		\begin{eqnarray*}
			h(x_k)&=&\sum_{i=1}^{m}\alpha_i h_{i}(x_k)= \sum_{i=1}^{m} \alpha_i k^{2i}\\
			&\geq& \alpha_mk^{2m}-\sum_{i=1}^{m-1}|\alpha_i|k^{2i}\geq \alpha_mk^{2m}+k^{2(m-1)}k\alpha_m \sum_{j=1}^{m}1\\
			&\geq& \alpha_m k^{2m-1}(k-m^2). 	
		\end{eqnarray*}
		Thus, $h\notin C(X)_{fin}^B$. Hence, $\text{span}(\mathbb{K})\cap C(X)_{fin}^B=\{0\}$. Consequently, the codimension of $C(X)_{fin}^B$ in $C(X)_{fin}^B\oplus M$ is infinite  as $C(X)_{fin}^B\oplus\text{span}(\mathbb{K}) \subseteq C(X)_{fin}^B\oplus M$. Which is not possible.\end{proof}
	
	Recall that an lcs $(Z, \tau)$ is \textit{quasibarreled} if each \textit{bornivorous} (a set which is absolutely convex, closed, and absorbs all bounded subsets of $Z$) is a neighborhood of $0$ (see, \cite{ hanslocally, tvsnarici}). 
	
	\begin{proposition}\label{quasibarreled} Suppose $(Z, \tau)$ is a quasibarreled lcs. Then the strong dual $(Z^*, \tau_s)$ is normable if and only if $(Z, \tau)$ is normable.
	\end{proposition}	
	\begin{proof} Suppose $(Z^*, \tau_s)$ is normable. Then its the strong dual $Z^{**}$ is normable. Since $(Z, \tau)$ is quasibarreled, $Z$ is isomorphic to its canonical image $J(Z)\subseteq Z^{**}$. Hence, $(Z, \tau)$ is normable.        \end{proof}
	
	Suppose $(Z, \tau)$ is an elcs. We denote $\tau_F$ by the finest locally convex topology (flc topology) for $(Z, \tau)$ which is still coarser than $\tau$. The space $(Z, \tau_F)$ is called the finest space for $(Z, \tau)$ (see, \cite{spoens, flctopology}).      
	
	\begin{theorem}\label{normability thm} Suppose $\mathcal{B}$ is a bornology on $X$ such that the finest space $(C(X), \tau_F)$ for $(C(X), \tau_{\mathcal{B}})$ is quasibarreled. Then the following statements are equivalent.
		\begin{enumerate}[$(1)$]
			\item 	The topology $\tau_{ucb}$ on the dual of $(C(X), \tau_{\mathcal{B}})$ is normable;
			\item $\tau_{ucb}$ on the dual of $(C(X), \tau_{\mathcal{B}}^s)$ is normable;
			\item $\tau_{\mathcal{B}}$ is normable on $C(X)$;
			\item $\tau_{\mathcal{B}}^s$ is normable on $C(X)$;
			\item $X$ is compact and $X\in \mathcal{B}$.  
		\end{enumerate}	
	\end{theorem}
	\begin{proof} The implications $(1)\Leftrightarrow(2)$ follow from Theorem 4.1 of \cite{Suc} and Proposition \ref{normability of strong topology}. The implications $(3)\Leftrightarrow (4)\Leftrightarrow(5)$ are easy to prove as for $B\in \mathcal{B}$, $\rho_{B}$ (or $\rho_{B}^s$) gives a norm on $C(X)$ if and only if $B=X$ and $B$ is compact.  
		
		$(1)\Rightarrow(3).$ By Proposition \ref{normability of strong topology}, $\mathcal{B}\subseteq \mathcal{K}$. Therefore, $\tau_{\mathcal{B}}=\tau_{F}$. Hence, by Proposition \ref{quasibarreled}, $\tau_{\mathcal{B}}$ is normable. 
		
		$(5)\Rightarrow(1)$. If $B=X$ is compact, then $\tau_\mathcal{B}$ is induced by the supremum norm. Consequently, its strong dual is normable.   	   
	\end{proof}
	
	\begin{corollary} For every metric space $X$, the strong dual of $(C(X), \tau_k)$ $((C(X), \tau_{p}))$ is normable if and only if $X$ is compact (finite).
	\end{corollary}
	\begin{proof}  For every metric space $X$, $(C(X), \tau_k)$ and $(C(X), \tau_{p})$ are always quasibarreled (see, Corollary 3 and Theorem 5, p. 234 in \cite{hanslocally}). It now follows from Theorem \ref{normability thm}.
	\end{proof}
	
	%
	
	%
	%
	%
	%
	
	We have noted in Proposition \ref{op forms a norm} that the operator norm may not define a norm on $(C(X), \tau_{\mathcal{B}})^*$ when $\mathcal{B}\supsetneq \mathcal{K}$. Therefore, we introduce a topology $\sigma$ on $\mathscr{M}_\mathcal{B}(X)$ that establishes a strong relationship with $((C(X), \tau_{\mathcal{B}})^*, \tau_{ucb})$.
	

	In the proof of Theorem \ref{decomposition of continuous linear functionals}, we have seen  that if $\mu\in\mathscr{M}_{\mathcal{B}}(X)$, then $$\|\mu\|=\sup\left\lbrace \int_X fd\mu: \|f\|_\infty\leq 1\right\rbrace.$$ If we replace $\left\lbrace f: \|f\|_\infty\leq 1\right\rbrace$ by any other bounded set $\mathbb{B}$ of $(C(X), \tau_{\mathcal{B}})$, then the map $\omega_\mathbb{B}(\mu)=\sup_{f\in \mathbb{B}}\left|\int_Xfd\mu\right|$ for $\mu\in\mathscr{M}_{\mathcal{B}}$ may attain the infinite value. However, if $D\in\mathcal{B}$ such that $C(X)_{fin}^D\subseteq C(X)_{fin}^B$ for all $B\in\mathcal{B}$, then for every bounded subset $\mathbb{B}$ of $C(X)_{fin}^D$, $\omega_\mathbb{B}$ defines a seminorm on $\mathscr{M}_{\mathcal{B}}$. These spaces are referred to in the literature as fundamental extended locally convex spaces. 
	
	\begin{definition}\normalfont (\cite{flctopology, esaetvs}) An elcs $(Z, \tau)$ is said to be a \textit{fundamental elcs} if $Z_{fin}$ is open in $(Z, \tau)$, where $$Z_{fin}=\bigcap\left\lbrace Z_{fin}^\rho: \rho \text{ is a continuous extended seminorm on } (Z, \tau)\right\rbrace.$$
	\end{definition}
	
	\begin{proposition} The following statements are equivalent for a bornology $\mathcal{B}$ on $X$.
		\begin{enumerate}[(1)]
			\item $(C(X), \tau_{\mathcal{B}})$ forms a fundamental elcs.
			\item There exists a $D\in\mathcal{B}$ such that $C(X)_{fin}^D\subseteq C(X)_{fin}^B$ for all $B\in\mathcal{B}$.
		\end{enumerate} 
	\end{proposition}
	\begin{proof} Suppose $D\in\mathcal{B}$ and $r>0$ such that $\rho_{D}^{-1}([0,r))\subseteq \bigcap_{B\in\mathcal{B}} C(X)_{fin}^B.$ It is now easy to see that $C(X)_{fin}^D\subseteq C(X)_{fin}^B$ for all $B\in\mathcal{B}$. Conversely, if $\rho$ is any continuous extended seminorm on $(C(X), \tau_{\mathcal{B}})$, then $\rho\leq \rho_B$ for some $B\in\mathcal{B}$. Consequently, $\bigcap_{B\in\mathcal{B}} C(X)_{fin}^B\subseteq C(X)_{fin}^\rho$ for all continuous extended seminorm  $\rho$ on $(C(X), \tau_\mathcal{B}).$ Thus, $(C(X), \tau_{\mathcal{B}})$ forms a fundamental elcs. 
	\end{proof}

	Now, suppose $\mathcal{B}$ is a bornology such that $(C(X), \tau_{\mathcal{B}})$ forms a fundamental elcs. Then there exists a $D\in\mathcal{B}$, such that $C(X)_{fin}^D\subseteq C(X)_{fin}^B$ for all $B\in\mathcal{B}$. Let $\sigma$ be the topology on $\mathscr{M}_{\mathcal{B}}(X)$ induced by the directed family $\{\omega_\mathbb{B}:\mathbb{B} \text{ is a bounded subset of } C(X)_{fin}^D\}$ of seminorms. Then $(\mathscr{M}_\mathcal{B}(X), \sigma)$ forms a locally convex space.

	\begin{theorem}\label{tauucb on the dual of C(X)} Suppose $\mathcal{B}$ is a bornology on $X$ such that $(C(X), \tau_{\mathcal{B}})$ forms a fundamental elcs. Then there exists a $D\in\mathcal{B}$ such that $((C(X), \tau_{\mathcal{B}})^*, \tau_{ucb})$  is isomorphic to product space $(\mathscr{M}_\mathcal{B}(X), \sigma)\times (M^*, \tau_{w^*})$, where $C(X)=C(X)_{fin}^D\oplus M$  and $\tau_{w^*}$ is the weak$^*$ topology on the dual $M^*$ of $M$.  \end{theorem}
	
	\begin{proof} Let $D\in\mathcal{B}$ such that $C(X)_{fin}^D\subseteq C(X)_{fin}^B$ for all $B\in\mathcal{B}$. Suppose $C(X)=C(X)_{fin}^D\oplus M$. Consider the map $$\Psi: (\mathscr{M}_\mathcal{B}(X), \sigma)\times (M^*, \tau_{w^*})\to ((C(X), \tau_{\mathcal{B}})^*, \tau_{ucb})$$
		defined by $$(\mu, H)\mapsto \int_{X} p_D(f)d\mu + H(q_D(f)) \text{ for } f\in C(X),$$ where $p_D$ and $q_D$ are the projections of $C(X)$ onto $C(X)_{fin}^D$ and $M$, respectively. Clearly, $\Psi$ is linear. If $\int_{X} p_D(f)d\mu + H(q_D(f))=0$ for all $f\in C(X)$, then $H=0$ and $\|\mu\|=0$ as $\|\mu\|=\sup\left\lbrace \left| \int_{X} fd\mu\right| : \|f\|_{\infty}\leq 1\right\rbrace$. Consequently, $\mu=0$ and $H=0$. Thus, $\Psi$ is injective. Since $C(X)_{fin}^D\oplus M$ is a topological direct sum and  $C(X)_{fin}^D\subseteq C(X)_{fin}^B$ for all $B\in\mathcal{B}$, by Theorem \ref{decomposition of continuous linear functionals}, $\Psi$ is surjective. 
		
		For the continuity of $\Psi$, let $\mathbb{B}$ be a bounded set in $(C(X), \tau_{\mathcal{B}})$. Clearly, $p_D(\mathbb{B})$ is a bounded subset of $C(X)_{fin}^D$.  Suppose $r>0$ and $\mathbb{F}$ is a finite subset of $M$ such that $\mathbb{B}\subseteq \rho_{D}^{-1}([0, r))+ \mathbb{F}$. Consequently, $q_D(\mathbb{B})$ is a finite subset of $M$.  Note that
		\begin{eqnarray*}
			\sup_{f\in \mathbb{B}} \left|\int_{X} p_D(f) d\mu+ H(q_D(f)) \right|&\leq& \sup_{f\in \mathbb{B}} \left|\int_{X} p_D(f) d\mu\right|+ \sup_{f\in \mathbb{B}}\left|H(q_D(f))\right|\\
			&\leq&	\sup_{f\in p_D(\mathbb{B})} \left|\int_{X} f d\mu\right|+ \sup_{f\in q_D(\mathbb{F})}\left|H(f)\right|\\
			&=& \omega_{p_D(\mathbb{B})}(\mu)+\rho_{q_D(\mathbb{B})}(H).     
		\end{eqnarray*}
		Hence, $\Psi$ is continuous. Now, let $\mathbb{B}$ and $\mathbb{F}$ be any bounded and finite subsets of $C(X)_{fin}^D$ and $M$, respectively. Then $\mathbb{B}\cup\mathbb{F}$ is bounded in $(C(X), \tau_{\mathcal{B}})$ and 
		$$\sup_{f\in \mathbb{B}}\left|\int_X fd\mu\right|+\sup_{f\in \mathbb{F}}|H(f)|\leq \sup_{f\in \mathbb{B}\cup \mathbb{F}}\left|\int_X p_D(f)d\mu+H(q_D(f))\right|.$$
		Therefore, $\Psi^{-1}$ is continuous. Hence, $\Psi$ is an isomorphism.
	\end{proof}
	
	Note that for every relatively compact set $B\subseteq X$, $\rho_B$ forms a seminorm on $C(X)$. Consequently, $M_B=\{0\}$. Therefore, for $\mathcal{B}\subseteq \mathcal{K}$, $\sigma$ is induced by $\{\omega_{\mathbb{B}}: \mathbb{B} \text{ is bounded in } (C(X), \tau_{\mathcal{B}})\}$. Thus, the following theorem follows from Theorem \ref{tauucb on the dual of C(X)}.  
	
	\begin{theorem} For every $\mathcal{B}\subseteq \mathcal{K}$, $(C(X), \tau_{\mathcal{B}})^*$ is isomorphic to $(\mathscr{M}_\mathcal{B}(X), \sigma)$. \end{theorem}
	
	\begin{remark}\normalfont If we take $X\in \mathcal{B}$ and $X$ is compact, then $\sigma$ and $\tau_{\mathcal{B}}$ will be induced by the total variation norm $\|\cdot\|$ and the supremum norm $\|\cdot\|_\infty$, respectively. 
	\end{remark}

	\begin{remark} Observe that for every $\mathcal{K} \subsetneq \mathcal{B}$, the operator norm $\|\cdot\|_{op}$ no longer defines a norm on $(C(X), \tau_{\mathcal{B}})^*$. Therefore, the topology $\sigma$ may serve as a suitable candidate for studying the functional-analytic properties of the function space $(C(X), \tau_{\mathcal{B}})$. 	\end{remark}

	\bibliographystyle{plain}
	\bibliography{reference_file}	
\end{document}